\documentclass{article}

\usepackage{cite}

\usepackage{amsmath, amssymb, verbatim, amsthm, graphicx, fullpage}

\newcommand{\C}{\mathbb{C}}
\newcommand{\N}{\mathbb{N}}
\newcommand{\Z}{\mathbb{Z}}
\newcommand{\Q}{\mathbb{Q}}
\newcommand{\F}{\mathbb{F}}

\newtheorem{conj}{Conjecture}

\newtheorem{lemma}{Lemma}
\newtheorem{thm}{Theorem}
\newtheorem{cor}{Corollary}

\begin{document}

\title{An application of a generalization of Artin's primitive root conjecture in the theory of monoid rings}

\author{R.\! C.\! Daileda}


\maketitle

\begin{abstract}
Using techniques of algebraic and analytic number theory, we resolve a question on monoid rings posed by Kulosman, et. al., under the assumption of the Generalized Riemann Hypothesis (GRH). Specifically, we show that under an appropriate GRH, for any (rational) prime $p$ the set $E(p)= \{q \text{ prime} \, | \, X^q-1 \text{ factors in } \F_p[X;M] \} $, where $M=\langle 2,3 \rangle = \N_0 \setminus \{ 1 \}$, contains a subset with positive natural density. In particular $E(p) \ne \varnothing$. This proves that $M$ is not a so-called ``Matsuda monoid'' of any positive type. For $p=2,3$ this was observed by Kulosman, who provided factorizations of $X^7-1$ and $X^{11}-1$ in $\F_2[X;M]$ and $\F_3[X;M]$, respectively. Our results explain and reproduce both of these factorizations, as well.
\end{abstract}

\section{Introduction}

Let $\Gamma$ be a commutative monoid, written additively.  We say that $\Gamma$ is {\em torison-free} if for every $n \in \N$ and every $\alpha, \beta \in \Gamma$, the equation $n \alpha = n \beta$ implies that $\alpha = \beta$.  If $\Gamma$ is torsion-free, $\alpha \in \Gamma$ and $n \in \N$, the equation $n x = \alpha$ has at most one solution in $\Gamma$.  If $px = \alpha$ has {\em no} solution in $\Gamma$ for every prime $p$, the authors of \cite{cgk} say that $\alpha$ has {\em height} $(0,0,\ldots)$.  If this is the case and $n \ge 2$ is an integer, choose any prime $p$ dividing $n$.  Then $nx = p(kx)$ for some $k \in \N$.  Since $px = \alpha$ has no solution in $\Gamma$, this implies that $nx = \alpha$ doesn't either.  We conclude that $\alpha$ has height $(0,0,\ldots)$ if and only if for every $n \ge 2$ the equation $n x = \alpha$ has no solution in $\Gamma$.  We prefer to term such an $\alpha$ to be {\em indivisible}.  One advantage of this formulation is that it has a simple positive expression:  $\alpha \in \Gamma$ is indivisible if and only if $n x = \alpha$ with $n \in \N$ and $x \in \Gamma$ implies $n=1$ and $x=\alpha$.

Given an additive monoid $\Gamma$ and a field $F$, let $F[X;\Gamma]$ denote the {\em monoid ring} of all polynomials in $X$ with exponents in $\Gamma$ and coefficients in $F$ (for a more careful definition and other essential properties of monoid rings, see \cite{gilmer}).  Given a nonzero $\pi \in \Gamma$, consider the binomial $X^{\pi} - 1 \in F[X;\Gamma]$.  If $\pi = n \alpha$ for some $n \ge 2$ and $\alpha \in \Gamma$, note that we have the factorization
\begin{equation}\label{first}
X^{\pi} - 1 = (X^{\alpha} - 1)\sum_{j=0}^{n-1} X^{j \alpha}.
\end{equation}
If $\Gamma$ is torsion-free and {\em cancellative} ($\alpha + \beta = \alpha + \gamma$ implies $\beta = \gamma$), \cite[Theorem 11.1]{gilmer} implies that neither factor on the right-hand side of \eqref{first} is a unit in $F[X;\Gamma]$, which shows that $X^{\pi}-1$ is reducible (not a unit and not irreducible) in $F [X; \Gamma]$.  Equivalently, if $X^{\pi}-1$ is irreducible in $F[X;\Gamma]$, then $\pi$ must be indivisible.  We are interested in conditions under which the converse holds.  Specifically, we say that a commutative, torsion-free, cancellative monoid $\Gamma$ is a {\em Matsuda monoid} provided that for every field $F$ and every indivisible $\pi \in \Gamma$, $X^{\pi}-1$ is irreducible in $F[X;\Gamma]$.  A somewhat weaker condition is that this condition hold for all fields $F$ of a given characteristic $p \ge 0$, and in this case we say that $\Gamma$ is a {\em Matsuda monoid of type $p$}.  

Matsuda \cite{matsuda} showed that every torsion free abelian group is a Matsuda monoid.  And according to \cite[Theorem 4.1]{cgk}, every submonoid of $\Q_0^{+}$ is a Matsuda monoid of type 0.  In positive characteristic, however, much less is known.  If
\[
M = \langle 2, 3 \rangle = \N_0 \setminus \{ 1 \}
\]
is the submonoid of $\N_0$ generated by $2$ and $3$, the indivisible elements of $M$ are precisely the prime numbers.  $M$ is {\em not} a Matsuda monoid of type 2 or 3.   This was observed in \cite{cgk}:  one has
\[
X^7 - 1 = (X^3 + X^2 + 1)(X^4 + X^3 + X^2 + 1)
\]
in $\F_2[X;M]$ and
\[
X^{11}-1 = (X^6 - X^5 - X^4 - X^3 + X^2 + 1)(X^5 + X^4 - X^3 + X^2 - 1)
\]
in $\F_3[X;M]$.  Although easily verified, these factorizations appear somewhat mysterious.  As we will see in Section \ref{easy}, they arise due to the fact that 2 and 3 are {\em not} primitive roots modulo 7 and 11, respectively.  We will explain and compute both factorizations in Section \ref{easy}.  

For an arbitrary positive (prime) characteristic $p$, we call a prime $q$ {\em exceptional} for $p$ if $X^q-1$ is reducible in $\F_p[X;M]$, and we let $E(p)$ denote the set of primes that are exceptional for $p$.  The set $E(p)$ characterizes the failure of the Matsuda property for $M$ in characteristic $p$.  It is natural, then, to wonder what can be said about the density of $E(p)$ in the set of all primes.  Our first main result shows that, subject to GRH, both $E(2)$ and $E(3)$ have positive natural densities.  

\begin{thm}\label{main1}
Let $p \in \{ 2, 3\}$ and assume that for every squarefree $n \in \N$ the Dedekind zeta function of $\Q(\sqrt[n]{1}, \sqrt[n]{p})$ satisfies the generalized Riemann hypothesis.  Then the set
\[
E(p) = \{q \text{ prime} \, | \, X^q - 1 \text{ reducible in } \F_p[X;M]  \} 
\]
has nonzero natural density $1-A$, where $A$ denotes Artin's constant
\[
\prod_{\ell \text{ prime} } \left( 1 - \frac{1}{\ell(\ell-1)} \right).
\]
\end{thm}

The techniques used to study $E(2)$ and $E(3)$ can't be applied to $E(p)$ in general.  But a modified approach allows us to treat the subset $E_2(p)$ consisting of those members of $E(p)$ for which $\langle p \rangle$ has index 2 in $(\Z/q\Z)^{\times}$.  Specifically, we will prove:

\begin{thm}\label{main2}
Let $p$ be an odd rational prime, and assume that for every odd squarefree $n \in \N$ the Dedekind zeta functions of the fields $\Q(\sqrt[n]{1}, \sqrt[n]{p})$ and $\Q(\sqrt[4n]{1}, \sqrt[4n]{p})$ satisfy the generalized Riemann hypothesis.  Then the set 
\[
E_2(p) = \big\{ q \text{ prime} \, \big| \, X^q - 1 \text{ is reducible in } \F_p[X;M]  \text{ and }  [(\Z/q\Z)^{\times}: \langle p \rangle] = 2 \big\}
\]
has nonzero natural density $a(p)$ given by
\[
a(p) = \frac{3}{4} \left( \frac{2p-1}{p^2 - p - 1}\right) A,
\]
where $A$ is Artin's constant.
\end{thm}

Taken together these results imply:

\begin{cor}
Assume that for every prime $p$ and every squarefree $n \in \N$, the Dedekind zeta functions of the fields $\Q(\sqrt[n]{1}, \sqrt[n]{p})$ and $\Q(\sqrt[4n]{1}, \sqrt[4n]{p})$ satisfy the generalized Riemann hypothesis.  Then $E(p)$ is infinite for all $p$.  In particular, the monoid $M = \langle 2, 3 \rangle$ is not a Matsuda monoid of any positive type.
\end{cor}

Our proof proceeds as follows.  In Section \ref{reverse} we show that the reducibility of $X^q-1$ in $\F_p[X;M]$ is equivalent to the vanishing of the trace coefficient of any factor of the image of the $q$th cyclotomic polynomial $\Phi_q$ in $\F_p[X]$.  In Section \ref{easy} we show that when $p=2,3$, {\em any} factorization of $\Phi_q$ in $\F_p[X]$ must necessarily have such a factor, which reduces the reducibility of $X^q-1$ in $\F_p[X;M]$ to the condition that $p$ fail to be a primitive root modulo $q$.  The (common) densities of $E(2)$ and $E(3)$ then follow at once from Artin's primitive root conjecture, which is known to be true under an appropriate GRH \cite{hooley}.  

When $p \ge 3$ the set $E(p)$ becomes more difficult to describe explicitly, so in Section \ref{fields} we partition $E(p)$ according to the value of the index $r=[(\Z/q\Z)^{\times} : \langle p \rangle]$, which when $r=2$ allows us to limit our search for ``traceless'' factors of $\Phi_q$ in $\F_p[X]$ to those factors that are also irreducible.  The existence of these types of factors is equivalent to a certain divisibility condition in the ring of integers of the $q$th cyclotomic field $K_q = \Q(\zeta_q)$.  When $r=2$ we are able to take advantage of the explicit description of the unique quadratic subfield of $K_q$ as $\Q(\sqrt{q^{\ast}})$ to convert this divisibility condition to the simple pair of congruences $q \equiv \pm 1 \!\! \pmod{4p}$.  This density of $E_2(p)$ is then given by a generalized version of Artin's conjecture \cite{lenstra, cw}, which is described in Section \ref{genartin}.  The nonvanishing of this density is the subject of Section \ref{apnonzero}.  Section \ref{remarks} includes computational evidence for our GRH conditional results, an example when $p=5$, and remarks on forthcoming work.  For other generalizations of Artin's original primitive root conjecture and their applications, see for instance \cite{ck, cw, cm, gold1, gold2, gold3, matt, moll, moree, murata, mp, q1, q2, sam, stadnik, wag, wein}.

\section{Reversal and Traces}\label{reverse}

Fix a prime $p$.  It will be convenient to view $\F_p[X;M]$ as a subring of $\F_p[X]$.  Given $g \in \F_p[X]$ we write
\[
g(X) = \sum_{j=0}^{m} a_j X^j
\]
with $m = \deg g$, so that $a_m \ne 0$.  Define the {\em reversal operator} $R : \F_p[X] \to \F_p[X]$ by
\[
R(g)(X)= \sum_{j=0}^m a_{m-j} X^j.
\]
Note that $\deg(R(g)) = \deg g$ if and only if $X \nmid g$, and in this case $R^2(g) = g$.  We also have
\[
R(g)(X) = X^{\deg g} g\left( \frac{1}{X}\right),
\]
which implies that $R$ is multiplicative.  

If $m=\deg g \ge 1$, we define the {\em trace} of $g$ to be
\[
\operatorname{Tr}(g) = - \frac{a_{m-1}}{a_m}.
\]
The reader can easily verify that for any nonconstant $g,h \in \F_p[X]$ and any $a \in \F_p^{\times}$ one has:
\begin{align}
\operatorname{Tr}(ag) &= \operatorname{Tr}(g); \label{T1}\\
\operatorname{Tr}(gh) &= \operatorname{Tr}(g) + \operatorname{Tr}(h). \label{T2}
\end{align}
Properties \eqref{T1} and \eqref{T2} suggest that we define $\operatorname{Tr}(a) = 0$ for any $a \in \F_p^{\times}$, but we prefer to leave the trace of a constant polynomial undefined.  The field theoretic trace is connected to the polynomial trace $\operatorname{Tr}(g)$ as follows.  If $g$ is irreducible in $\F_p[X]$ and $\alpha$ is any root of $g$ in an extension of $\F_p$, then
\begin{equation}\label{tt}
\operatorname{Tr}(g) = \operatorname{Tr}_{\F_p[\alpha] / \F_p} (\alpha).
\end{equation}
If $\deg g = f$, then $\F_p[\alpha]/\F_p$ is a Galois extension of degree $f$ whose Galois group is generated by the {\em Frobenius automorphism} $\alpha \mapsto \alpha^p$.  By \eqref{tt} we therefore have
\begin{equation}\label{tt1}
\operatorname{Tr}(g) = T_{p,f}(\alpha),
\end{equation}
where
\[
T_{p,f}(X) = \sum_{j=0}^{f-1} X^{p^j} \in \F_p[X].
\]

We will say that $g \in \F_p[X]$ is {\em traceless} if $\operatorname{Tr}(g) = 0$.  By definition a traceless polynomial must have positive degree.  Let
\[
\begin{aligned}
\mathcal{F}_1 &= \{ g \in \F_p[X;M] \, : \, g \text{ is nonconstant and } X \nmid g \}, \\
\mathcal{F}_2 &= \{ g \in \F_p[X] : g \text{ is traceless and } X \nmid g \}.
\end{aligned}
\]
It is easy to see that $R : \mathcal{F}_1 \to \mathcal{F}_2$ is an involution.  For any prime $q$ we have
\[
R(X^q - 1) = 1 - X^q = - (X^q - 1).
\]
Because $X \nmid X^q - 1$ and $R$ is multiplicative, it follows that $X^q - 1$ factors in $\mathcal{F}_1$ if and only if it factors in $\mathcal{F}_2$.  This proves:

\begin{lemma}\label{t1}
For any primes $p$ and $q$, $X^q - 1$ is reducible in $\F_p[X;M]$ if and only if it factors as the product of traceless polynomials in $\F_p[X]$.
\end{lemma}

Because $X^p-1 = (X-1)^p$ in $\F_p[X]$, the nontrivial (monic) factors of $X^p-1$ in $\F_p[X]$ all have the form $(X-1)^k$ for some $1 \le k <p$. Since \eqref{T2} implies $\operatorname{Tr}((X-1)^k) = k \operatorname{Tr}(X-1) = k$, we find that $X^p-1$ {\em is not} the product of traceless polynomials in $\F_p[X]$.  Therefore $X^p-1$ is irreducible in $\F_p[X;M]$ by Lemma \ref{t1}, so that $p \not \in E(p)$.  Similarly, $2 \not \in E(p)$ for any $p$.  Hence, when trying to determine whether or not $q \in E(p)$, we may assume that $q$ is odd and different from $p$.

Let
\[
\Phi_q(X) = \sum_{j=0}^{q-1}X^j
\]
denote the $q$th cyclotomic polynomial, which can be considered as a member of $\F_p[X]$ for any $p$.  Suppose $\Phi_q = gh$ in $\F_p[X]$ with $g$ traceless.  Multiplying by $X-1$ we obtain $X^q - 1 = g \widetilde{h}$.  Because $X^q-1$ and $g$ are both traceless, so is $\widetilde{h}$ by \eqref{T2}.  Applying Lemma 1 we conclude that if $\Phi_q$ has a traceless factor in $\F_q[X]$, then $X^q-1$ factors nontrivially in $\F_q[X;M]$.  Conversely, if $X^q-1$ is reducible in $\F_p[X;M]$, then by Lemma \ref{t1} we can write $X^q-1 = gh$ with $g, h \in \F_q[X]$ traceless.   Because $X-1$ is prime in $\F_p[X]$ and divides $X^q-1$, without loss of generality we have $h = (X-1) \tilde{h}$ for some $\tilde{h} \in \F_p[X]$.  Cancelling the common factor of $X-1$ we obtain $\Phi_q = g \tilde{h}$ with $g$ traceless.  This proves:

\begin{lemma}\label{t2}
For any primes $p \ne q$, $X^q-1$ is reducible in $\F_p[X;M]$ if and only if $\Phi_q$ has a traceless factor in $\F_p[X]$.
\end{lemma}

Write $\Phi_q = g_1 \cdots g_r$ with $g_j \in \F_p[X]$ irreducible.  In order for $\Phi_q$ to have a traceless factor in $\F_p[X]$, it is necessary and sufficient that there exist a nonempty $J \subset \{ 1, \ldots , r \}$ so that $\prod_{j \in J} g_j$ is traceless.  It will therefore be useful to recall a few basic facts about the irreducible factorization of $\Phi_q$ in $\F_p[X]$.  For $q \ne p$, let $f(p,q)$ denote the order of the subgroup $\langle p \rangle$ generated by the image of $p$ in $(\Z/q\Z)^{\times}$.   Then it is well known that $\Phi_q$ has exactly $r(p,q) = \frac{q-1}{f}=[(\Z/q\Z)^{\times}:\langle p \rangle]$ distinct irreducible factors in $\F_p[X]$, each of multiplicity 1 and degree $f$.  It follows that if $\zeta$ is a root of any $g_j$ in an extension of $\F_p$, then $[\F_p[\zeta] : \F_p] = f$.  The extension $\F_p[\zeta] / \F_p$ is necessarily Galois and its Galois group is generated by the {\em Frobenius automorphism} $\alpha \mapsto \alpha^p$.  By equation \eqref{tt1} we therefore have
\begin{equation}\label{tt3}
\operatorname{Tr}(g_j) =  T_{p,f}(\zeta).
\end{equation}

\section{The Densities of $E(2)$ and $E(3)$}\label{easy}

Because the fields $\F_2$ and $\F_3$ are particularly small, we can give an alternate characterization of the sets $E(2)$ and $E(3)$ which will make it possible to compute their GRH-conditional densities using Artin's (original) primitive root conjecture and work of Hooley \cite{hooley}.  For $a \in \N$ let $\mathcal{S}(a)$ denote the set of primes for which $a$ is a primitive root.  We will prove:

\begin{lemma}\label{boom}
We have
\[
E(2) = \mathcal{P} - \mathcal{S}(2) \,\,\, \text{ and } \,\,\, E(3) = \mathcal{P} - \mathcal{S}(3).
\]
\end{lemma}

In other words, $E(2)$ consists of those primes for which 2 {\em fails} to be a primitive root, and likewise for $E(3)$.  According to Artin's primitive root conjecture, for every prime $p \not \equiv 1 \pmod{4}$ one expects to have $d(\mathcal{S}(p)) = A$, where
\[
A = \prod_{\ell \text{ prime}} \left( 1 - \frac{1}{\ell (\ell - 1)}\right) = 0.3739 \ldots
\]
is {\em Artin's constant}.  This would mean that
\begin{equation}\label{d2}
d(E(2)) = d(E(3)) = 1-A = 0.6260\ldots,
\end{equation}
a conclusion which is in close agreement with the computationally obtained results
\[
\frac{\# E(2) \cap [1,10^6]}{\# \mathcal{P} \cap [1,10^6]} = 0.6262\ldots \,\,\, \text{ and } \,\,\, \frac{\# E(3) \cap [1,10^6]}{\# \mathcal{P} \cap [1,10^6]} = 0.6255\ldots
\]
Hooley \cite{hooley} showed that when $a \ne \pm 1, \square$, Artin's conjecture for $\mathcal{S}(a)$ is true, provided one assumes that the zeta functions of certain Kummerian extensions of $\Q$ obey GRH.  A careful reading of \cite{hooley} shows that when $a=p$, it is sufficient to assume GRH only for the zeta functions of the fields $\Q(\sqrt[n]{1}, \sqrt[n]{p})$, where $n \in \N$ is squarefree.  It is now a simple matter to deduce Theorem \ref{main1}.

\begin{proof}[Proof of Theorem \ref{main1}]
In view of Lemma \ref{boom} and the remarks above, it suffices only to observe that $A<1$.
\end{proof}

To prove Lemma \ref{boom}, we begin with a simple observation.  If $q$ is prime and $\Phi_q = gh$ with $g,h \in \F_2[X]$ of positive degree, then without loss of generality $g$ must be traceless.  This is because we have
\[
1 = -1 = \operatorname{Tr}(\Phi_q) = \operatorname{Tr}(g) + \operatorname{Tr}(h)
\]
by \eqref{T2}, and we are working in $\F_2$.  This means that $\Phi_q$ has a traceless factor in $\F_2[X]$ if and only if $\Phi_q$ is reducible in $\F_2[X]$.  Since the number of irreducible factors of $\Phi_q$ is precisely $\frac{q-1}{f}$, the latter occurs if and only if $f < q-1$.  That is, $\Phi_q$ is reducible in $\F_2[X]$ if and only if $2$ is {\em not} a primitive root modulo $q$.  Thus, by Lemma \ref{t2},
\[
E(2) = \mathcal{P} -\mathcal{S}(2).
\]

In order to handle $E(3)$ we introduce another lemma.

\begin{lemma}\label{t5}
Let $p \ne q$ be primes and let $a \in \F_p$.  Then every irreducible factor of $\Phi_q$ in $\F_p[X]$ has trace equal to $a$ if and only if $\Phi_q$ is irreducible in $\F_p[X]$ and $a=-1$.
\end{lemma}

\begin{proof}
Since $\operatorname{Tr}(\Phi_q) = -1$, just the ``only if'' implication requires proof.  As above, write $\Phi_q = g_1 \cdots g_r$ with each $g_i \in \F_p[X]$ irreducible of degree $f$ and $r = \frac{q-1}{f}$.  Suppose that $\operatorname{Tr}(g_j) = a$ for all $j$.  If $\zeta \ne 1$ is a primitive $q$th root of unity over $\F_p$, then $\Phi_q(\zeta) = 0$, which implies $g_j(\zeta) = 0$ for some $j$.  Then by \eqref{tt3} we have
\[
a =  \overline{T_{p,f}}(\zeta),
\]
where $\overline{T_{p,f}} \in \F_p[X]$ is the polynomial obtained by replacing each exponent in $T_{p,f}$ with its least positive residue modulo $q$.  Therefore, $\zeta$ is a root of $\overline{T_{p,f}} - a \in \F_p[X]$, and hence $g_j$ divides $\overline{T_{p,f}} - a$ in $\F_p[X]$.  Because $\zeta$ was arbitrary and the $g_j$ are distinct, it follows that $\Phi_q$ divides $\overline{T_{p,f}} -  a$ in $\F_p[X]$.  But $\deg(\Phi_q)=q-1$ and $\deg (\overline{T_{p,f}} -a) \le q-1$, so we must have $\Phi_q = \overline{T_{p,f}} - a$.  Therefore $a=-1$ and 
\[
\sum_{j=1}^{q-1} X^j = \sum_{k=0}^{f-1} X^{p^k \!\!\!\!\!\! \pmod{q}}.
\]
Comparing the monomials that appear on each side of this equality, we see that the powers of $p$ modulo $q$ must exhaust the set $\{1, 2, \ldots , q-1 \}$.  That is, $p$ is a primitive root modulo $q$, so that $f=q-1$ and $\Phi_q$ is irreducible in $\F_p[X]$.  This completes the proof.
\end{proof}

Now let $q \ne 3$ be a prime for which 3 fails to be a primitive root.  Then $\Phi_q$ has $r = \frac{q-1}{f} > 1$ irreducible factors in $\F_3[X]$.  If one of them is traceless, then $q \in E(3)$ by Lemma \ref{t3}.  Otherwise, by Lemma \ref{t5}, there must be distinct irreducible factors $g_1, g_2 \in \F_3[X]$ of $\Phi_q$ satisfying $\operatorname{Tr}(g_1)=1$ and $\operatorname{Tr}(g_2)=-1$.  Then $g_1 g_2$ is a traceless factor of $\Phi_q$ in $\F_3[X]$, and $q \in E(3)$ by Lemma \ref{t2}. This shows that $E(3) = \mathcal{P} - \mathcal{S}(3)$ and finishes the proof of Lemma \ref{boom}.

Using Lemma \ref{boom} one quickly finds that
\[
\begin{split}
E(2) &= \{7, 17, 23, 31, 41, 43, 47, 71, 73, 79, 89, 97, \ldots \},\\
E(3) &= \{11, 13, 23, 37, 41, 47, 59, 61, 67, 71, 73, 83, 97, \ldots \}.
\end{split}
 \]
For a prime $q$ in either of these sets, the proofs of Lemmas \ref{t1}, \ref{t2} and \ref{boom} yield an algorithm for producing the corresponding factorization of $X^q-1$.  When $p=2$, we need only factor $\Phi_q$ in $\F_2[X]$, multiply by $X-1$ and reverse coefficients.  For example, when $q=7$ we have
\begin{equation}\label{phi7}
\Phi_7(X) = (X^3 + X^2 + 1)(X^3 + X+1).
\end{equation}
Because the second factor is traceless, we multiply both sides by $X-1$, absorb it into the first factor and obtain
\[
X^7-1 = (X^4 + X^2 + X + 1)(X^3 + X + 1).
\]
Finally, reversing coefficients (there's no need to ``negate'' since $1=-1$ in $\F_2$) we have
\begin{equation}\label{ex7}
X^7 - 1 = (X^3 + X^2 + 1)(X^4 + X^3 + X^2 + 1)
\end{equation}
in $\F_2[X;M]$, which is precisely the factorization appearing in \cite{cgk}.  Since $f(2,7)=3$ and $r(2,7)=2$, the factors in \eqref{ex7} are irreducible in $\F_2[X]$.  Because only one of these factors is traceless, it follows that \eqref{ex7} is the {\em only} factorization of $X^7-1$ in $\F_2[X;M]$ (up to the order in which the factors are written).  This demystifies the appearance of \eqref{ex7} in \cite{cgk}:  7 is simply the smallest prime for which 2 is not a primitive root, and there is only one traceless factorization of $X^7-1$ in $\F_2[X]$.

When $p=3$ and $q=11$ we have $f=5$ and $r=2$.  Factoring $\Phi_{11}$ over $\F_3$ yields
\[
\Phi_{11}(X) = (X^5 + X^4 - X^3 + X^2 -1 )(X^5 - X^3 + X^2 - X -1).
\]
The second factor is again traceless, so we multiply by $X-1$ and absorb it into the first factor:
\[
X^{11}-1 = (X^6 + X^4 - X^3 - X^2 - X + 1)(X^5 - X^3 + X^2 - X -1)
\]
in $\F_3[X]$.  Reversing coefficients and negating we have
\[
X^{11}-1 = (X^6 - X^5 - X^4 - X^3 + X^2 + 1)(X^5 + X^4 - X^3 + X^2 - 1)
\]
in $\F_3[X;M]$.  As in the preceding example, 11 is the smallest prime for which 3 fails to be a primitive root, and this is essentially the only factorization of $X^{11}-1$ in $\F_3[X;M]$ possible, which provides an explanation for its appearance in \cite{cgk}.

The smallest $q$ for which none of the irreducible factors of $\Phi_q$ in $\F_3[X]$ is traceless is $q=67$.  Here $f = 22$, $r=3$ and $\Phi_{67} = g_1 g_2 g_3$ in $\F_3[X]$ with $\operatorname{Tr}(g_1)=\operatorname{Tr}(g_2) = -1$ and $\operatorname{Tr}(g_3)=1$ (the interested reader can easily compute $g_1, g_2$ and $g_3$ explicitly with the help of her favorite computer algebra software).  Following the proof of Lemma \ref{boom}, we see that $g_2 g_3$ is a traceless factor of $\Phi_{67}$ over $\F_3$.  Multiplying by $(X-1)$, reversing coefficients and negating, we see that
\begin{equation}\label{ex67}
X^{67} - 1 = ((X-1)R(g_1))R(g_2 g_3)
\end{equation}
is a factorization of $X^{67}-1$ in $\F_3[X;M]$.  Because $X-1$, $g_1$, $g_2$ and $g_3$ are irreducible in $\F_3[X]$ but not traceless, \eqref{ex67} is a factorization into irreducibles in $\F_3[X;M]$.   Pairing $g_1$ with $g_3$ instead yields {\em different} irreducible factorization, which incidentally proves that $\F_3[X;M]$ is {\em not} a unique factorization domain.

\section{Number Fields}\label{fields}

When $p>3$, the condition $f>1$ is {\em not} sufficient to guarantee membership in $E(p)$.  For instance, $f(5,11)=5$ so that $r(5,11)=2$, but neither of the $2$ irreducible factors of $\Phi_{11}$ in $\F_5[X]$ is traceless.  So $11 \not \in E(5)$ by Lemma \ref{t2}.  We therefore partition $E(p)$ according to the value of the index $r$.  For $k \in \N$ set
\[
E_k(p) = \{q \text{ prime} \, | \, q \in E(p) \text{ and } r(p,q) = k \}. 
\]
When $r(p,q)=2$, notice that $q \in E(p)$ if and only if $\Phi_q$ has a traceless {\em irreducible} factor in $\F_p[X]$.  Regardless of the value of the index, according to \eqref{tt3} the latter statement holds if and only if there is a $q$th root of unity $\zeta \ne 1$ in an extension of $\F_p$ which satisfies $T_{p,f}(\zeta) = 0$.  Because the nontrivial $q$th roots of unity are precisely the roots of $\Phi_q$, this can happen if and only if $T_{p,f}$ and $\Phi_q$ have a nontrivial (positive degree) common factor in  $\F_p[X]$.  Together with Lemma \ref{t2} this proves:

\begin{lemma}\label{t3}
For any primes $p \ne q$, $\Phi_q$ has a traceless irreducible factor in $\F_p[X]$ if and only if $\operatorname{gcd}(\Phi_q,T_{p,f}) \in \F_p[X]$ is nontrivial.  In this case $X^q-1$ is reducible in $\F_p[X;M]$.
\end{lemma}

Let $G_{p,q} = \operatorname{gcd}(\Phi_q,T_{p,f}) \in \F_p[X]$.  We now see that
\[
E_2(p) = \{ q \text{ prime} \, | \, r(p,q) = 2 \text{ and }  \deg G_{p,q} > 0  \}.
\]
When $p$ is odd and $r(p,q) \ge 3$, it is {\em not} generally the case that $q \in E(p)$ implies $\deg G_{p,q}>0$, as our example with $p=3$ and $q=67$ demonstrates.  It is interesting to note that when $p=2$, however, we have:

\begin{cor}
Let $q$ be an odd prime.  Then $X^q - 1$ is reducible in $\F_2[X;M]$ if and only if $\gcd(\Phi_q,T_{2,f}) \in \F_2[X]$ is nontrivial.
\end{cor}

\begin{proof}
Because of Lemma \ref{t3}, it suffices to prove the ``only if'' implication.  Suppose $X^q - 1$ is reducible in $\F_2[X;M]$.  Then $\Phi_q$ is reducible in $\F_2[X]$ by Lemma \ref{t2}, and Lemma \ref{t5} implies that $\operatorname{Tr}(g) \ne 1$ for some irreducible factor $g \in \F_2[X]$ of $\Phi_q$.  Because the only other possible value for the trace is 0, the conclusion now follows from Lemma \ref{t3}.
\end{proof}

Our goal now is to find an alternate characterization of the condition $\deg G_{p,q} > 0$ in terms of algebraic number theory.  We assume throughout that $p$ is odd.  Although we have already considered the case $p=3$, there is no disadvantage to including it here.  The condition that $\Phi_q$ and $T_{p,f}$ have a nontrivial common factor over $\F_p$ is equivalent to the statement that $T_{p,f}$ fails to be a unit in the quotient ring $\F_p[X]/(\Phi_q)$.  Let $\zeta_q \in \C$ denote a primitive $q$th root of unity and set $K_q = \Q(\zeta_q)$.  Let $\mathcal{O}_q = \Z[\zeta_q]$ denote the ring of integers in $K_q$.  Because $\Phi_q$ is the minimal polynomial of $\zeta_q$ over $\Q$, we have natural isomorphisms
\[
\F_p[X]/(\Phi_q) \cong \left((\Z/p\Z)[X] \right) / (\Phi_q) \cong \Z[X]/(p,\Phi_q) \cong \left( \Z[X]/(\Phi_q)\right) /(p) \cong \mathcal{O}_q /(p),
\]
under which (the coset of) $X$ corresponds to (the coset of) $\zeta_q$.  We immediately find that $T_{p,f}$ is not a unit in $\F_p[X]/(\Phi_q)$ if and only if $\alpha_q=T_{p,f}(\zeta_q)$ is divisible by one of the primes of $\mathcal{O}_q$ lying over $p$.  This occurs if and only if $p$ divides the (ideal) norm $\mathcal{N}_q(\alpha_q \mathcal{O}_q)= \left| \mathcal{O}_q/\alpha_q \mathcal{O}_q \right|$.

We can go a bit further if we introduce some Galois theory.  The extension $K_q/\Q$ is Galois with group isomorphic to $(\Z/q\Z)^{\times}$.  Explicitly, an element $a \in (\Z/q\Z)^{\times}$ corresponds to the automorphism induced by $\zeta_q \mapsto \zeta_q^a$.  If $a_1, \ldots, a_r$ are coset representatives for $\langle p \rangle$ in $(\Z/q\Z)^{\times}$, where $r = \frac{q-1}{f}$, then the $T_{p,f}(\zeta_q^{a_i})$ are the distinct conjugates of $\alpha_q$ over $\Q$.  It follows that
\[
m_q(X) = \prod_{i=1}^r (X - T_{p,f}(\zeta_q^{a_i})) \in \Z[X]
\]
is the minimal polynomial for $\alpha_q$ over $\Q$ (the coefficients are integral because $\alpha_q$ is an algebraic integer).  Up to a sign, the constant coefficient of $m_q(X)$ is equal to the {\em (algebraic) norm}
\[
N_q(\alpha_q) = N_{\Q(\alpha_q)/\Q}(\alpha_q) = \prod_{i=1}^r T_{p,f}(\zeta_q^{a_i}).
\]
Because $[K_q : \Q(\alpha_q)] = \frac{q-1}{r} = f$, the two notions of norm are related by the equation
\[
\mathcal{N}_q(\alpha_q \mathcal{O}_q) = N_q(\alpha_q)^f.
\]
This means that $p$ divides $\mathcal{N}_q(\alpha_q \mathcal{O}_q)$ if and only if $p$ divides $N_q(\alpha_q)$.  The preceding paragraph and Lemma \ref{t3} now yield:

\begin{lemma}\label{t4}
For any odd primes $p \ne q$, $\deg G_{p,q} > 0$ if and only if $p|N_q(\alpha_q)$. 
\end{lemma}

Lemmas \ref{t2}, \ref{t3} and \ref{t4} yield
\[
E_2(p) = \{ q \text{ prime} \, | \, r(p,q)=2 \text{ and }  p| N_q(\alpha_q) \}.
\]
If $q \in E_2(p)$, then $\alpha_q$ has a single distinct Galois conjugate $\beta_q$ and
\begin{equation}\label{trace}
\alpha_q + \beta_q = \sum_{a \in (\Z/q\Z)^{\times}} \zeta_q^a = -1,
\end{equation}
since $\Phi_q(\zeta_q)=0$.  Equation \eqref{trace} has two important consequences.  Because $\mathcal{O}_q = \Z[\zeta_q]$ and $\Q(\alpha_q) = K_q^{\langle p \rangle}$, it is not hard to show that the ring of integers in $\Q(\alpha_q) = \Q(\alpha_q,\beta_q)$ is $\Z[\alpha_q, \beta_q] = \Z[\alpha_q]$, by \eqref{trace}.  This means that $\operatorname{disc}(\Q(\alpha_q)) = \operatorname{disc}(\alpha_q)$. Furthermore, the minimal polynomial for $\alpha_q$ over $\Q$ is
\[
m_q(X) = (X-\alpha_q)(X-\beta_q) = X^2 + X + N_q(\alpha_q),
\]
which means that
\[
\operatorname{disc}(\alpha_q) = \operatorname{disc}(X^2 + X + N_q(\alpha_q)) = 1 - 4 N_q(\alpha_q).
\]

For odd $q$, the unique quadratic subfield of $K_q$ is known to be $\Q\left(\sqrt{q^{\ast}} \right)$, where $q^{\ast} = (-1)^{\frac{q-1}{2}} q$.  Since $\Q(\alpha_q)$ also has degree 2 over $\Q$, we must have $\Q(\sqrt{q^{\ast}}) = \Q(\alpha_q)$.  In particular
\[
q^{\ast} = \operatorname{disc}  \Q\left(\sqrt{q^{\ast}} \right)  = \operatorname{disc} \Q(\alpha_q) = \operatorname{disc}(\alpha_q) = 1 - 4 N_q(\alpha_q).
\]
Therefore $p | N_q(\alpha_q)$ if and only if $q^{\ast} \equiv 1 \pmod{4p}$.  Since $p$ is odd this congruence is equivalent to the system $q^{\ast} \equiv 1 \pmod 4$ and $q^{\ast} \equiv 1 \pmod{p}$, by the Chinese remainder theorem.  The condition $q^{\ast} \equiv 1 \pmod{4}$ is automatic, while $q^{\ast} \equiv 1 \pmod{p}$ is equivalent to
\begin{equation}\label{gummo}
q \equiv \pm 1 \!\!\! \pmod{4p}.
\end{equation}
We have therefore proven:

\begin{lemma}\label{yup}
For any odd prime $p$, 
\[
E_2(p) = \left\{ q \text{ prime} \, \Big| \, q \equiv \pm 1 \!\!\! \pmod{4p}, f =   \frac{q-1}{2}  \right\}.
\]
\end{lemma}

The condition $\frac{q-1}{f}=2$ implies that $\left( \frac{p}{q} \right) = 1$, and one can use quadratic reciprocity to show that this implies $q \equiv \pm 1 \pmod{4p}$ when $p=2,3$.  The expression for $E_2(p)$ given in Lemma \ref{yup} is therefore valid for {\em all} primes $p$.

\section{$E_2(p)$ and a generalized Artin conjecture}\label{genartin}

We continue to assume that $p$ is odd.  The set $E_2(p)$ falls into a wider class of sets of primes whose densities were investigated by Lenstra in \cite{lenstra}.   Such sets arose in the work of several authors \cite{cw, q1, q2, sam, wein} in the 1970s, in connection with certain properties of euclidean rings in global fields. Although these sets of primes (places) can be constructed in an arbitrary global field $K$, for our purposes it is sufficient to simply take $K=\Q$.

Given a prime number $q$ and an integer $n$, let $\nu_{q}^{\times} : \Q \to \Z$ denote the $q$-adic valuation (with $\nu_q(0) = \infty$), and let
\[
R_q = \{ a \in \Q \, | \, \nu_{q}(a) \ge 0 \}  \,\,\, \text{ and } \,\,\, R_q^{\times} = \{a \in \Q \, | \, \nu_q(a)=0 \} 
\]
be the valuation ring and unit group at $q$, respectively.  There is a natural surjection $\psi: R_q^{\times} \to (\Z/q\Z)^{\times}$ obtained by composing reduction modulo $qR_q$ with the isomorphism $R_q /qR_q \cong \Z/q\Z$.  For a subgroup $W \subseteq R_{q}^{\times}$, we let $r_{q}(W)$ denote the index of $\psi(W)$ in $(\Z/q\Z)^{\times}$.  

If $F$ is a finite Galois extension of $\Q$, $C \subseteq G(F/\Q)$ is a union of conjugacy classes, $W \subset \Q^{\times}$ is a finitely generated subgroup with positive rank, and $k$ is a positive integer, let $\mathcal{S}(F,C,W,k)$ denote the set of primes $q$ so that
\[
\begin{aligned}
\left( \frac{F/\Q}{q} \right) \subseteq C, \,\,\,\, W \subseteq R_{q}^{\times}, \,\,\,\,\text{and} \,\,\,\, r_{q}(W) | k,\\[5pt]
\end{aligned}
\]
where $\left( \frac{F/\Q}{q} \right)$ is the Frobenius symbol.  The density of $\mathcal{S}(F,C,W,k)$ is the subject of the following conjecture, which is a combination of \cite[Conjecture (2.3)]{lenstra} and \cite[(2.15)]{lenstra}.

\begin{conj}[Lenstra]\label{density}
Under the hypotheses stated above, the natural density of $\mathcal{S}(F,C,W,k)$ exists and is equal to
\begin{equation}\label{dense}
\delta(F,C,W,k) = \sum_n \frac{\mu(n) c(n)}{[F \cdot L_{n} : \Q]},
\end{equation}
where for squarefree $n$, $L_n = \Q(\zeta_{q(n)},W^{1/q(n)})$ with $q(n) = \prod_{\ell | n} \ell^{\nu_{\ell}(k)+1}$, and $c(n) = \# (C \cap G(F/F\cap L_n))$. 
\end{conj}

Because $\mathcal{S}(a) = \mathcal{S}(\Q, \{1\}, \langle a \rangle, 1)$, Conjecture \ref{density} subsumes Artin's primitive root conjecture, and is therefore commonly referred to as a {\em generalized Artin conjecture}.  Various other generalizations exist; see for example \cite{ck, fm, mp, stadnik, wag}.  We remark that because Lenstra \cite{lenstra} works with sets of primes in an arbitrary global field, his results are all actually stated in terms of {\em Dirichlet} density.  However, if we are only working with primes in a number field, then we are free to use natural density in place of Dirichlet density (see \cite[p. 203]{lenstra}).

Set $F = \Q(\zeta_{4p})$.  Identify $G(F/\Q)$ with $(\Z/4p\Z)^{\times}$ in the usual manner, mapping $a \in (\Z/q\Z)^{\times}$ to the automorphism $\sigma_a$ defined by the rule $\zeta_{4p} \mapsto \zeta_{4p}^a$.  We claim that
\begin{equation}\label{another}
E_2(p) = \mathcal{S}(F,\{ \pm 1\}, \langle p \rangle , 2).
\end{equation}
Since $\left( \frac{F/\Q}{q} \right) = \{ q \}$ and $r_q(\langle p \rangle) = \frac{q-1}{f}$, the only question is why membership in $\mathcal{S}(F,\{ \pm 1\}, \langle p \rangle, 2)$ guarantees that $\frac{q-1}{f}=2$.  To that end, let $q \in \mathcal{S}(F,\{ \pm 1 \},\langle p \rangle,2)$.  Since $p$ is odd this implies $q^{\ast} \equiv 1 \pmod{p}$, and quadratic reciprocity then gives
\[
\left( \frac{p}{q} \right) = \left( \frac{q^{\ast}}{p} \right) = 1.
\]
We conclude that $\langle p \rangle$ is a subgroup of the group of squares in $(\Z/q\Z)^{\times}$.  Because the squares have index 2 in $(\Z/q\Z)^{\times}$, this implies $\frac{q-1}{f}$ is divisible by 2.  But we also have that $r_q(\langle p \rangle)$ divides 2.  Since $r_q(\langle p \rangle) = \frac{q-1}{f}$, we are finished.

Let $a(p) = \delta(\Q(\zeta_{4p}),\{ \pm 1 \}, \langle p \rangle, 2 )$.  We then have:

\begin{lemma}\label{bluey}
If Conjecture \ref{density} is true, then for any odd prime the set $E_2(p)$ has natural density $a(p)$.  
\end{lemma}

We will return to the truth of Conjecture \ref{density} following the next section.

\section{Nonvanishing of $a(p)$}\label{apnonzero}

Note that Conjecture \ref{density} only asserts that the density of $\mathcal{S}(F,C,W,k)$ {\em exists}.  Whether or not $\delta(F,C,W,k)$ is {\em nonzero} must be addressed separately.  The following result allows us to answer this question without explicitly computing $\delta(F, C, W, k)$.

\begin{lemma}[{\cite[Theorem (4.1)]{lenstra}}]\label{nonzero}
Let $h$ be the product of the primes $\ell$ so that $W \subseteq (\Q^{\times})^{ q(\ell)}$.  Then $\delta(F,C,W,k) \ne 0$ if and only if there is a $\sigma \in G(F(\zeta_h)/\Q)$ so that
\begin{align}
&\sigma|F \in C, \label{c1} \\[4pt]
&\sigma|L_{\ell} \ne \operatorname{id} \text{ whenever } L_{\ell} \subset F(\zeta_h). \label{c2}
\end{align}
\end{lemma}

Because it is a potential source of confusion, we point out that the variable $p$ appearing in the statement of the original \cite[Theorem (4.1)]{lenstra} represents the characteristic of the global field in question, which in our case is simply 0.  It is not ``our'' $p$.  

\begin{lemma}\label{fuckthis}
For any odd prime $p$ one has $a(p) \ne 0$.
\end{lemma}

\begin{proof}
Taking $F=\Q(\zeta_{4p})$, $C= \{ \pm 1\}$, $W= \langle p \rangle$ and $k=2$ in Lemma \ref{nonzero}, we have $h=1$, since $p$ is not an $m$th power in $\Q$ for any $m \ge 2$.  So $F(\zeta_h) = F$ and from \eqref{c1} and \eqref{c2} we find that $a(p) \ne 0$ if and only if there is a $\sigma \in \{ \pm 1 \}$ so that 
\begin{equation}\label{good}
\sigma|L_{\ell} \ne \operatorname{id} \text{ whenever } L_{\ell} \subseteq F.
\end{equation}
However, $L_{\ell}$ contains the subextension $\Q(p^{1/\ell})$, which fails to be Galois over $\Q$.  Since $F/\Q$ is abelian, this implies $L_{\ell}$ is {\em never} contained in $F$.  So \eqref{good} is vacuously satisfied for any $\sigma \in G(F/\Q)$, and hence $a(p) \ne 0$.  
\end{proof}

We can also show that $a(p)$ is nonzero directly, by expressing it as the product of a certain nonzero rational number and Artin's constant $A$.  In \cite{wag} Wagstaff obtained similar expressions for the densities of the sets
\[
\mathcal{S}(a ,t) = \{ q \text{ prime} \, | \,  \langle a \rangle \text{ has index } t \text{ in } (\Z/q\Z)^{\times}   \},
\]
to which Conjecture \ref{density} also applies (see \cite[p.\ 216]{lenstra}).  When $a>0$ or $a \ne \pm \square$, Wagstaff was also able to give an Euler product expansion for $d(\mathcal{S}(a,t))$.  Later Moree \cite{moree} extended Wagstaff's results to obtain Euler products for every $d(\mathcal{S}(a,t))$.  However $E_2(p)\ne \mathcal{S}(p,2)$ when $p \ge 5$, since in $\mathcal{S}(p,2)$ one has no control over the residue classes of its members.  So the results of Moree and Wagstaff can't help us compute $a(p)$.  Instead, we find a closed form expression for the series \eqref{dense} defining $a(p) = \delta(\Q(\zeta_{4p}),\{ \pm 1 \}, \langle p \rangle, 2)$.

\begin{lemma}\label{ap}
For any odd prime $p$,
\begin{equation}\label{product}
a(p) = \frac{3(2p-1)}{4(p^2-p-1)}A,
\end{equation}
where $A$ is Artin's constant.  In particular $a(p) \ne 0$.
\end{lemma}

\begin{proof}
With $F=\Q(\zeta_{4p})$, $C= \{ \pm 1\}$, $W = \langle p \rangle$ and $k=2$ in Section 6, we have
\[
L_n = \begin{cases} \Q(\zeta_n, p^{1/n}) & \text{ if } 2 \nmid n,\\ \Q(\zeta_{2n},p^{1/2n}) & \text{ otherwise,} \end{cases}
\]
so that
\[
c(n) =  \begin{cases} 2 & \text{ if } F \cap L_n \subseteq \mathbb{R}, \\1 & \text{ otherwise} \end{cases} = \begin{cases} 2 & \text{ if }(2p,n)=1, \\ 1 & \text{ otherwise.}\end{cases}
\]
Taking advantage of \cite[Theorem VI.9.4]{lang} we also find that
\[
[F \cdot L_n : \Q] = \begin{cases} 2 n \varphi(n) (p-1) & (p,n)=1,\\[2pt] 2 n \varphi(n) & p|n \text{ and } 2 \nmid n,\\[2pt] 4 n \varphi(n) & 2p|n. \end{cases}
\]
It then follows from \eqref{dense} that
\[
a(p) = \frac{1}{2} \sum_{p|n \atop 2 \nmid n} \frac{\mu(n)}{ n \varphi(n)}+ \frac{1}{4}\sum_{2p|n} \frac{\mu(n)}{ n \varphi(n)} +\frac{1}{2(p-1)} \sum_{p \nmid n \atop 2 | n} \frac{\mu(n)}{n \varphi(n)} + \frac{1}{p-1} \sum_{(2p,n)=1} \frac{\mu(n)}{n \varphi(n)}.
\]
Because $\mu$ is multiplicative,
\[
\sum_{(2p,n)=1} \frac{\mu(n)}{n \varphi(n)} = \prod_{\ell \text{ prime} \atop \ell \ne 2,p} \left( 1 - \frac{1}{\ell(\ell-1)}\right) = 2 \frac{p(p - 1)}{p^2 - p - 1} A.
\]
Therefore
\[
\frac{1}{2} \sum_{p|n \atop 2 \nmid n} \frac{\mu(n)}{ n \varphi(n)} = \frac{1}{2} \sum_{(2p,m)=1} \frac{\mu(mp)}{ m\, p \, \varphi(mp)} = - \frac{1}{p(p-1)} \frac{p(p - 1)}{p^2 - p - 1} A = - \frac{1}{p^2 - p - 1} A.
\]
Likewise
\[
\frac{1}{4}\sum_{2p|n} \frac{\mu(n)}{ n \varphi(n)} = \frac{1}{4} \sum_{(2p,m)=1} \frac{\mu(2pm)}{2pm \varphi(2pm)} = \frac{1}{4p(p-1)}  \frac{p(p - 1)}{p^2 - p - 1} A  = \frac{1}{4(p^2 - p - 1)} A.
\]
Finally,
\[
\frac{1}{2(p-1)}\sum_{p \nmid n \atop 2 | n} \frac{\mu(n)}{n \varphi(n)} = \frac{1}{2(p-1)} \sum_{(2p,m)=1} \frac{\mu(2m)}{2m \varphi(2m)} = -\frac{1}{2(p-1)}  \frac{p(p - 1)}{p^2 - p - 1} A = - \frac{p}{2(p^2 - p - 1)} A.
\]
Formula \eqref{product} follows immedately.
\end{proof}

\begin{proof}[Proof of Theorem \ref{main2}]  According to \cite[Theorem (3.1)]{lenstra}, the generalized Artin Conjecture \ref{density} is true under the assumption of the generalized Riemann hypothesis for the Dedekind zeta functions of the number fields $L_n$.  And in this case Lemma \ref{ap} implies that for any odd prime $p$ one has
\[
d(E_2(p)) = a(p) = \frac{3(2p-1)}{4(p^2 - p - 1)}A \ne 0.
\]
\end{proof}

\section{Remarks}\label{remarks}

Empirical evidence appears to support our GRH-conditional result that $d(E_2(p)) = a(p)$.  Table \ref{data} compares the approximate density $\# (E_2(p) \cap [1,10^6]) / \#( \mathcal{P} \cap [1,10^6])$ to $a(p)$, for the first ten values of $p$.  It is interesting to note that for $p<100$, the percentage error in the approximate density is substantially larger for $p=3$ than it is for other primes.  This appears to persist as we increase the size of the primes included in the approximation.  It might be interesting to know the source of this phenomenon.

\begin{table}[h]
\caption{Approximate and conjectural densities of $E_2(p)$}
\label{data}
\begin{center}
\begin{tabular}{c|c|c}
$p$ & $\pi_p(10^6)/\pi(10^6)$ & $a(p)$ \\[2pt] \hline 
2 & 0.28143 & 0.28046\\[2pt]
3 & 0.30052 & 0.28046\\[2pt]
5 & 0.13815 & 0.13285\\[2pt]
7 & 0.09112 & 0.08892 \\[2pt]
11 & 0.05461 & 0.05403 \\[2pt]
13 & 0.04635 & 0.04523 \\[2pt]
17 & 0.03448 & 0.03415 \\[2pt]
19 & 0.03076 & 0.03043 \\[2pt]
 23 & 0.02563 & 0.02499 \\[2pt]
 29 & 0.01949 & 0.01971 
\end{tabular}
\end{center}
\end{table}

For $p > 3$, determining membership in $E(p)$ in general appears to be somewhat challenging.  The condition $\deg \gcd(\Phi_q,T_{p,q}) >0 $ (remember that the GCD is being computed over $\F_p$) is easy enough to test with a computer algebra system (once one recognizes that $\gcd(\Phi_q,T_{p,q} ) = \gcd(\Phi_q,\overline{T_{p,q}})$).  But as we have seen, this will only detect those $q \in E(p)$ for which $\Phi_q$ has a traceless {\em irreducible} factor, and these $q$ comprise a proper subset $E'(p)$ of the set $E(p)$ in which we are we are interested.  It would be interesting to know the precise extent to which $E'(p)$ and $E(p)$ differ.  

To that end, for $r \ge 2$ let 
\[
I_r(p) = \{ q \text{ prime } \, | \, [(\Z/q\Z)^{\times}: \langle p \rangle] = r \},
\] 
so that $E_r(p) \subseteq I_r(p)$.  A simple combinatorial argument shows that for any fixed $p$, $I_r(p) = E_r(p)$ for all sufficiently large $r$.  So the set 
\[
F(p) = \{ r \ge 2 \, | \, E_r(p) \subsetneq I_r(p) \}
\]
is finite and by our results contains $r=2$ for $p \ge 5$.  Since
\[
E(p) = \left( \bigcup_{r \in F(p)} E_r(p) \right) \cup \left(\bigcup_{r \not \in F(p)} I_r(p) \right) ,
\]
it would therefore be useful to determine $F(p)$ completely, and this seems necessary in order to achieve a more thorough understanding of $E(p)$ for $p \ge 5$.  

For example, one can show (using techniques generalizing those in Section \ref{fields}) that 
\[
I_3(5) = E_3(5),
\]
proving that $3 \not \in F(5)$, and computational evidence suggests that $r \not \in F(5)$ for all $r \ge 3$.  This would mean that $F(5) = \{2\}$ and hence
\[
E(5) = E_2(5) \cup \bigcup_{r \ge 3} I_r(5),
\]
making computation of the density of $E(5)$ amenable to the generalized Artin Conjecture \ref{density}.  Our computations have also shown that $3 \in F(7)$, and also suggest that $r \not \in F(7)$ for $r \ge 4$.  This yields the analogous conjectural decomposition
\[
E(7) = E_2(7) \cup E_3(7) \cup \bigcup_{r \ge 4} I_4(7),
\]
which can be used to compute a GRH-conditional value for $d(E(7))$ as well, provided one can also find a ``nice'' description of $E_3(7)$, along the lines of that given for $E_2(p)$ in Lemma \ref{yup}.  Determining the set $F(p)$ for $p \ge 5$, and describing the sets $E_r(p)$ for $r \in F(p)$, are indeed the subjects of our continuing research.

\section{Acknowledgements}

I am indebted to my friend Nathan Jones of the University of Illinois at Chicago for the suggestion to decompose $E(p)$ according to the index $r$.  I am also extremely grateful to Carl Pomerance for useful comments on an earlier draft of this paper which greatly improved our exposition here.

\end{document}